\documentclass[a4paper]{article}
\makeatletter

\newcommand{\Rmnum}[1]{\expandafter\@slowromancap\romannumeral #1@}
\makeatother
\usepackage{graphicx}
\usepackage{amsfonts}
\usepackage{amsmath}
\usepackage{amssymb}
\usepackage{fancyhdr}
\usepackage{titlesec}
\usepackage{indentfirst}
\usepackage{tikz}
\usepackage{tikz-cd}
\usepackage{booktabs}
\usepackage{verbatim}
\usepackage{color}
\usepackage{amsthm}
\usepackage{subfigure}
\usepackage{abstract}
\usepackage{multirow}
\usepackage{fullpage}
\usepackage{hyperref}
\usepackage[margin=0pt]{geometry}
\usepackage{lipsum}
\usepackage{layout}
\usepackage{mathtools}
\usepackage[page,header]{appendix}
\usepackage{titletoc}
\usepackage{newfloat}
\usepackage{csquotes}
\numberwithin{equation}{section}
\newtheorem{theorem}{Theorem}[section]
\newtheorem{lemma}[theorem]{Lemma}
\newtheorem{example}[theorem]{Example}

\newtheorem{definition}[theorem]{Definition}
\newtheorem{remark}[theorem]{Remark}

\newtheorem{conjecture}{Conjecture}

\newcommand{\RNum}[1]{\uppercase\expandafter{\romannumeral #1\relax}}

\DeclareMathOperator{\disc}{disc}
\DeclareMathOperator{\Gal}{Gal}

\DeclareMathOperator{\Disc}{Disc}
\DeclareMathOperator{\Nm}{Nm}

\DeclareMathOperator{\ord}{ord}

\DeclareMathOperator{\Cl}{Cl}
\DeclareMathOperator{\Rg}{Rg}

\DeclareMathOperator{\Coker}{Coker}

\DeclareMathOperator{\Frob}{Frob}
\DeclareMathOperator{\fp}{\mathfrak{p}}

\usepackage[OT2,T1]{fontenc}
\DeclareSymbolFont{cyrletters}{OT2}{wncyr}{m}{n}
\DeclareMathSymbol{\Sha}{\mathalpha}{cyrletters}{"58}

\newcommand{\zp}{\mathbb{Z}/p\mathbb{Z}}

\newcommand{\Q}{\ensuremath{{\mathbb{Q}}}}

\DeclareFloatingEnvironment[fileext=dia,within=section]{diagram}

\usepackage{pdfsync}

\begin{document}
\title{Pointwise Bound for $\ell$-torsion in Class Groups II: \\
	Nilpotent Extensions}

\author{Jiuya Wang}

\newcommand{\Addresses}{{
		\bigskip
		\footnotesize		
		Jiuya Wang, \textsc{Department of Mathematics, Duke University, 120 Science Drive 117 Physics Building Durham, NC 27708, USA
		}\par\nopagebreak
		\textit{E-mail address}: \texttt{wangjiuy@math.duke.edu}	
	}}
\maketitle	
	\begin{abstract}
	 For every finite $p$-group $G_p$ that is non-cyclic and non-quaternion and every positive integer $\ell\neq p$ that is greater than $2$, we prove the first non-trivial bound on $\ell$-torsion in class group of every $G_p$-extension. More generally, for every nilpotent group $G$ where every Sylow-$p$ subgroup $G_p\subset G$ is non-cyclic and non-quaternion, we prove a non-trivial bound on $\ell$-torsion in class group of every $G$-extension for every integer $\ell>1$. 
	\end{abstract}
	
\bf Key words. \normalfont $\ell$-torsion conjecture, nilpotent group, descendant tree of $p$-groups
\pagenumbering{arabic}	
\section{Introduction}
This is a sequel paper of the author \cite{JW20} on the following conjecture. 
\begin{conjecture}[$\ell$-torsion Conjecture]\label{conj:l-torsion}
	Given an integer $\ell>1$ and a number field $k$. For any degree $d$ extension $F/k$, the size of $\ell$-torsion in the class group of $F$ is bounded by
	$$|\Cl_{F}[\ell]| = O_{\epsilon,k}(\Disc(F)^{\epsilon}). $$
\end{conjecture}
This conjecture has been brought forward previously by \cite{BruSil,Duk98,Zha05}. We refer the audience to the first paper \cite{JW20} for an introduction of its relations to many other questions in arithmetic statistics (including Malle's conjecture, Cohen-Lenstra heuristics, integral points and Selmer groups of curves).

By a theorem of Brauer-Siegel, see for example \cite{SL}, the class number of $F$ with $[F:\mathbb{Q}]=d$ is bounded by $O_{\epsilon,d}(\Disc(F)^{1/2+\epsilon})$. This gives the so-called \emph{trivial bound} for $\ell$-torsion in class groups:
\begin{equation}\label{eqn:trivial-bound}
|\Cl_F[\ell]|  = O_{\epsilon}(\Disc(F)^{1/2+\epsilon}).
\end{equation}
Although Conjecture \ref{conj:l-torsion} proposes a bound as small as $O_{\epsilon, k}(\Disc(F)^{\epsilon})$, in terms of what can be really proved, it is still wildly open to break $1/2$ into $1/2-\delta$ where $\delta>0$ is an arbitrarily small positive number. We will call such a bound a \emph{non-trivial bound} for $\ell$-torsion in class groups. 

We now give a brief summary on progress towards Conjecture \ref{conj:l-torsion}. Firstly, we mention all cases where Conjecture \ref{conj:l-torsion} is currently known, that is, $\ell$-torsion in class groups of all $\ell$-extensions over an arbitrary number field $k$, see e.g. \cite[section 2]{KluWan} for a compact treatment. This includes, for example, Gauss's classical results on $2$-torsion for quadratic extensions, which is usually considered the only case where Conjecture \ref{conj:l-torsion} is achieved. These results all essentially come from a direct use of genus theory, and was mentioned in previous literatures on isolated small degree cases. 

Other than special cases from genus theory, results on this question are basically categorized into two directions: conditional result assuming GRH and unconditional result. The most broad conclusion is due to Ellenberg-Venkatesh \cite{EV07} where a non-trivial bound in the order of $O_{\epsilon, k}( \Disc(F)^{1/2- 1/2\ell(d-1)+\epsilon})$ is proved upon assuming GRH for  Artin L-functions. On the other hand, for the unconditional result, we know much less. When $\ell = 2$, \cite{BSTTTZ} gives  a non-trivial bound for $2$-torsion in class groups for all extensions by using geometry of numbers. In the same work of Ellenberg-Venkatesh \cite{EV07}, using reflection principle, an unconditional result for $\ell = 3$ for all small degree extensions with $d\le 4$ is given. Earlier results on $\ell=3$ for quadratic extensions can also be found in \cite{Pie05,HV06}. For every $\ell>3$, in the author's previous paper \cite{JW20}, we show a non-trivial bound for $\ell$-torsion in class groups of number fields where Galois group is $G= (\zp)^r$ with $r>1$. 

We mention that there are also recent results, see e.g. \cite{Ellen16,ML17,Widmer1,FreWid18,Chen,FreWid18x,ZTArtin} on removing the GRH condition in \cite{EV07} and get a non-trivial bound on $\ell$-torsion in class groups \emph{on average}. In contrast to results on average, this paper, together with the first paper \cite{JW20} in this sequence, focuses on proving results for \textit{every} extension while removing the GRH assumption. Notice that for most cases treated in this paper (aside from $\ell=2$ and the special cases $\ell=p$), even an average result hasn't been worked out before. 

Comparing to the previous paper \cite{JW20} where we focus on very restricted Galois groups, in this paper we enlarge the set of Galois groups where a non-trivial point-wise bound holds unconditionally for arbitrary $\ell>1$ to a much more general family of groups. For an arbitrary integer $\ell>1$, we denote $\mathcal{G}_k(\ell)$ to be the set of permutation Galois groups $G$ where there exists $\delta_k(G, \ell)>0$ such that $|\Cl_F[\ell]| =  O_{\epsilon, G,k}( \Disc(F)^{1/2 - \delta(G, \ell)+\epsilon})$ for every $G$-extension $F/k$. We write $\mathcal{G}(\ell)$ and $\delta(G,\ell)$ in short when $k = \Q$. In this language, aside from special cases that can be handled by genus theory, we know that: by \cite{BSTTTZ}, the group $G\in \mathcal{G}(2)$ for every transitive permutation group $G\subset S_n$; by \cite{EV07}, the group $G\in \mathcal{G}(3)$ if $G\subset S_n$ is a transitive permutation group with degree $n\le 4$; by \cite{JW20}, for general $\ell$, the group $A\in \mathcal{G}(\ell)$ for all elementary abelian groups $A = (\zp)^r$ with $r>1$. Our main theorem is to greatly enlarge the set $\mathcal{G}_k(\ell)$ for every $\ell$.

Our main theorem is as follows. A group $G$ is \emph{non-quaternion} if $G$ is not a generalized quaternion group, see \ref{sssec:quaternion} for more details on generalized quaternion groups. 
\begin{theorem}\label{thm:main}
	The regular representation of every non-cyclic and non-quaternion $p$-group $G_p$  is in $\mathcal{G}(\ell)$ for every integer $\ell>1$. More generally, the regular representation for every nilpotent group $G$ is in $\mathcal{G}(\ell)$ for every integer $\ell>1$ if its Sylow-$p$ subgroup $G_p$ is in $\mathcal{G}(\ell)$ for every $p||G|$. 
\end{theorem}
\begin{remark}
   An analogue of Theorem \ref{thm:main} also holds over arbitrary base field $k$. All results are effective. 
\end{remark}
For example, Theorem \ref{thm:main} proves that every non-cyclic abelian $p$-group $A_p$ is in $\mathcal{G}(\ell)$. This largely generalizes the previous result of the author, which we can rephrase as follows:
\begin{theorem}[\cite{JW20}, Theorem $1.1$]\label{thm:base-case}
    Every non-cyclic elementary abelian group $A$ is in $\mathcal{G}(\ell)$. 
\end{theorem}

The proof of Theorem \ref{thm:base-case} heavily relies on a result from representation theory, i.e., for every non-cyclic elementary abelian group $A$, we have for every $A$-extension $L/k$, 
$$|\Cl_{L/k}[\ell]| = \prod_{K_i/k} |\Cl_{K_i/k}[\ell]|,\quad\quad \Disc(L/k)= \prod_{K_i/k} \Disc(K_i/k),$$
where $K_i/k$ ranges over all degree $p$ subfields of $L$, see \cite[section $3$]{JW20} for more details. However such a nice structure does not hold for general finite groups $G$. In particular, one can show that among all abelian groups, elementary abelian groups are the only groups carrying such a nice structure. Although such a rigid group structure is the key reason in obtaining good upper bounds on $\ell$-torsion in \cite{JW20} for some cases, it largely limits the Galois groups where the method applies. As a contrast, in this paper, we will develop a much softer way, so that we can prove results for a much broader set of Galois groups. 

The main strategy of this work is two-sided. 
\begin{itemize}
	\item 
	Firstly, we introduce a new type of group extension, we call it \emph{forcing extension}, see Definition \ref{def:forcing}. We develop an Extension Lemma \ref{lem:ind-hyp-2}, which is an induction on Lemma \ref{lem:EV} in \cite{EV07}, specially for forcing extensions. This enables us to deduce the $\ell$-torsion bound for a large degree number field from the $\ell$-torsion bound for a small degree number field. More precisely, if $G\in \mathcal{G}_k(\ell)$, then for a forcing extension $\pi$, 
		\begin{center}
			\begin{tikzcd}
				0\arrow{r} & H\arrow[r] & \tilde{G} \arrow[r,"\pi"] & G \arrow{r} & 0,
			\end{tikzcd}
		\end{center}	
	we prove that $\tilde{G}$ is also in $\mathcal{G}_k(\ell)$. 
	
	\item
	Secondly, we prove that every non-cyclic and non-quaternion $p$-group can be constructed via iterated forcing extensions from its Frattini quotient. This enables us to apply the Extension Lemma proved before to all such $p$-groups, with the initial cases $G_p = (\zp)^r$ ($r>1$) in Theorem \ref{thm:base-case} proved in \cite{JW20}. This requires a careful analysis on the composition series of a fixed finite $p$-group. 
\end{itemize}

The organization of the paper is following. In section \ref{sec:induction}, we prove two induction lemmas: in section \ref{ssec:induction-extension}, we give the Extension Lemma for an inductive use of \cite{EV07} on forcing group extensions; in section \ref{ssec:induction-compositum}, we give the Compositum Lemma for an inductive use of \cite{EV07} on compositum of extensions. In section \ref{sec:group-theory}, we focuses on studying $p$-groups, and we prove that every non-cyclic and non-quaternion $p$-group can be constructed via iterated forcing group extensions. Finally in section \ref{sec:main-proof}, we give the proof for the main theorem.

\section{Notations}
\noindent
$k$: a number field considered as the base field\\
%$|\cdot|$: the absolute norm $\Nm_{k/\Q}$\\
$\tilde{F}/k$: Galois closure of $F$ over $k$\\
$\Gal(F/k)$: Galois group of $F/k$ as a permutation group\\
$Z(G)$: center of a finite group $G$\\
$G^{ab}$: abelianization of a finite group $G$\\
$\Disc(F/k)$: absolute norm of relative discriminant $\Nm_{k/\Q}(\disc(F/k))$ of $F/k$ where $\disc(F/k)$ is the relative discriminant ideal in $k$, when $k=\mathbb{Q}$ it is the usual absolute discriminant\\
$\Cl_{F/k}$: relative class group of $F/k$, when $k=\mathbb{Q}$ it is the usual class group of $F$\\
$\Cl_{F/k}[\ell]$:  $\{ [\alpha]\in \Cl_{F/k}\mid \ell [\alpha] = 0\in \Cl_{F/k} \}$\\
$|\Cl_{F/k}[\ell]|$, $|\Cl_{F}[\ell]|$: the size of $\Cl_{F/k}[\ell]$, $\Cl_{F}[\ell]$\\
%$\pi(Y; q, a)$: the number of prime numbers $p$ such that $p< Y$ and $p\equiv a \mod q$\\
$\pi(Y; L/k, \mathcal{C})$: for a Galois extension $L/k$, the number of unramified prime ideals $p$ in $k$ with $|p|<Y$ and $\Frob_{p} \in \mathcal{C}$ where $\mathcal{C}$ is a conjugacy class of $\Gal(L/k)$ \\
%$\pi(Y; L/k, \hat{\mathcal{C}})$: the number of unramified prime ideals $p$ in $L$ with $|p|<Y$ and $\Frob_{p} \notin \mathcal{C}$ where $\mathcal{C}$ is a conjugacy class of $\Gal(L/k)$\\
%$A\asymp B$: there exist absolute constants $C_1$ and $C_2$ such that $C_1 B\le A\le C_2 B$\\
$\Delta(\ell,d)$: a constant slightly smaller than $\frac{1}{2\ell(d-1)}$, for details see Lemma \ref{lem:EV} \\

%Warning: In order to simplify the notation for the whole paper, unless specifically mentioned otherwise, the implied constants $O_{\epsilon}$, $O_{\epsilon,k}$, $O_{\epsilon, k, \epsilon_0}$ will always depend on $\ell, d$ aside from the dependence indicated in the symbol when we are stating results or conjectures on bounding $\ell$-torsion in class groups of degree $d$ extensions.\\
 
\section{Induction over Ellenberg-Venkatesh}\label{sec:induction}
In this section, we are going to prove two versions of inductive methods to apply the following critical lemmas by Ellenberg-Venkatesh \cite{EV07}.
\begin{lemma}[\cite{EV07}]\label{lem:EV}
	Given a Galois extension $L/K$ and $0< \theta< \frac{1}{2\ell(d-1)}$ and an integer $\ell>1$, denote $$M:=\pi(\Disc(L/K)^{\theta};L/K,e),$$ then 
	\begin{equation}\label{eqn:EV-lemma}
	|\Cl_{L}[\ell]|= O_{\epsilon,[K:\Q],\ell}\Big(\frac{\Disc(L)^{1/2+\epsilon}}{M}\Big).
	\end{equation}
\end{lemma}

As one can observe, a critical input in applying Lemma \ref{lem:EV} is a good estimate for $M$. It would be fantastic if we have a good lower bound on the value of $M$ in terms of $\Disc(L)$. However, the exact challenge comes from the condition we impose on $\theta$, that is, $\theta$ need to be really small. In particular, the bound $\Disc(L/K)^{\theta}$ by which we count prime ideals is so small that no current versions of effective Chebotarev density theorem can guarantee a single prime that is split in $L/K$ without assuming GRH. If we are allowed to use GRH, then the effective Chebotarev density theorem proven by \cite{LO75} immediately give a good lower bound on $M$.
\begin{theorem}[\cite{LO75}, Effective Chebotarev Density Theorem on GRH]\label{thm:GRH}
Given a Galois extension $L/K$ with Galois group $G$. Assuming GRH, then for every $x\ge 2$, we have
	$$\big|   \pi(x; L/K, e) - \frac{1}{|G|} \textnormal{Li}(x) \big| = O_{[L:\Q]}(x^{1/2}\ln (\Disc(L) x) ).$$
\end{theorem}
As a corollary, assuming GRH, we can take $\theta = \Delta(\ell, d)$ in Lemma \ref{lem:EV}, i.e., arbitrarily close to $\frac{1}{2\ell(d-1)}$, and then get $M = \Disc(L/K)^{\Delta(\ell,d)}$. 

In fact, if we do not assume GRH, \cite{LO75} also proves an unconditional result which requires the bound $x \ge \exp(10[L:\Q] (\ln \Disc(L))^2)$ to be at least sub-exponential in $\Disc(L)$. It is of course too far away from the allowable range in Lemma \ref{lem:EV}. 

In this paper we will show how to apply Lemma \ref{lem:EV} to get a pointwise saving by unconditional knowledge on distribution of prime ideals. Actually it suffices if we can count prime ideals where the range $x$ is a polynomial in $\Disc(L)$. We will apply the following statement in our proof since the format of the statement is convenient for us to give a uniform treatment for a large family of groups all at once.

 \begin{lemma}[\cite{May, ZamThesis}]\label{lem:MZ}
 	Given $L/k$ a Galois extension of number fields with $[L:\Q]=d$. There exists absolute, effective constants $\gamma = \gamma(k, G)>2$, $\beta = \beta(k, G)>2$, $D_0 = D_0(k)>0$ and $C=C(k)>0$ such that if $\Disc(L/k)\ge D_0$, then for $x\ge \Disc(L/k)^{\beta}$, we have
 	$$\pi(x;L/k, \mathcal{C})  \ge C_k\frac{1}{\Disc(L/k)^{\gamma}}\cdot \frac{|\mathcal{C}|}{|G|} \cdot \frac{x}{\ln x} .$$
 \end{lemma}
\begin{remark}
	The actual values for $\beta$ and $\gamma$ are determined in \cite{May} when $k=\Q$ and $L/k$ is abelian. The actual values for general cases are also determined in \cite{ZamThesis}. We leave them as a symbol since these numbers could potentially be improved in the future. In all cases, the value for $\beta$ is much larger than $\gamma$. So we will always assume $\beta> \gamma+1/2$ in this paper. The reason we make this assumption is to simplify the numerical analysis in the proof of Extension Lemma \ref{lem:ind-hyp-2}. 
\end{remark}
 We mention that results in this direction have also appeared previously in \cite{Weiss, Deb,ZTleast,ZTBrunT,MontVau}. We expect that other versions of statements in this type (including both upper and lower bounds) can also be applied to some groups in our argument, and will result in different amount of power savings in the final answer. For example, in the author's previous paper \cite{JW20}, an upper bound result in \cite{MontVau} seems to give optimal savings among all results in this direction. However, since we do not aim to optimize the savings in this work, we will simply apply Lemma \ref{lem:MZ} by which we can get a uniform proof.

\subsection{Induction by Group Extension}\label{ssec:induction-extension}
In this section, our main goal is to prove Extension Lemma \ref{lem:ind-hyp-2}. We first define a new type of group extension \textit{forcing extensions}.
\begin{definition}[Forcing Extension]\label{def:forcing}
	We say that a group extension $(\tilde{G}, \pi)$ of $G$
		\begin{center}
			\begin{tikzcd}
				0\arrow{r} & H\arrow[r] & \tilde{G} \arrow[r,"\pi"] & G \arrow{r} & 0,	
			\end{tikzcd}
		\end{center}
		is \emph{forcing} if there exists a conjugacy class $\mathcal{C}\subset G$ such that for every element $c\in \mathcal{C}$, all elements in $\pi^{-1}(c)\subset \tilde{G}$ has the same order with $c\in G$. We will also say that $(\tilde{G}, \pi)$ is \emph{forcing with respect to $\mathcal{C}$}.		
\end{definition}

\begin{remark}
This notion of forcing extension does not fit with other common seen concepts of group extensions. Split extensions are not necessarily forcing, and central extensions are not necessarily forcing. For example, the cyclic group $C_6$, as a group extension of $C_3$ by $C_2$, is both split and central, but it is not forcing. However, we will see that this notion of forcing extension is particularly amenable to discussions on $p$-groups. 
\end{remark}

\begin{lemma}[Extension Lemma]\label{lem:ind-hyp-2}
	Given two finite groups $G$ and $H$ with $|G|=n$ and $|H|=m$ and an arbitrary integer $\ell>1$. Suppose the regular representation of $G$ is in $\mathcal{G}_k(\ell)$ with respect to $\delta = \delta_k(G,\ell)$. If the extension $(\tilde{G}, \pi)$ of $G$
	\begin{center}
		\begin{tikzcd}
			0\arrow{r} & H\arrow[r] & \tilde{G} \arrow[r,"\pi"] & G \arrow{r} & 0.	
		\end{tikzcd}
	\end{center}
	is a forcing extension with respect to $\mathcal{C}$, then the regular representation $\tilde{G}$ is in $\mathcal{G}_k(\ell)$ with respect to $$\delta_k(\tilde{G},\ell) = \delta_k(G,\ell) \cdot \eta_0$$
	where $\eta_0=\frac{\Delta(\ell, m)}{m\cdot \Delta(\ell,m) + r\cdot\max\{ \beta, \gamma \}}$ and $r = \ord(g)$ for $g\in \mathcal{C}$.
\end{lemma}

We first give a lemma on the size of $\ell$-torsion in relative class groups for a general $\ell>1$. 
\begin{lemma}\label{lem:trivial-bound-relative}
	Given a relative extension $L/F/k$ and an arbitrary integer $\ell>1$, we have
	$$\frac{|\Cl_{L/k}[\ell]|}{|\Cl_{F/k}[\ell]|} \le  |\Cl_{L/F}[\ell]| \le |\Cl_{L/F}| \le [L:F] \cdot \frac{|\Cl_L|}{|\Cl_F|}\le [L:F] \cdot \frac{\Disc(L)^{1/2}}{\Disc(F)^{1/2}}.$$
\end{lemma}
\begin{proof}
	By the definition of relative class group, there exists a subgroup $N = \Nm_{L/F}(\Cl_L)\subset \Cl_F$ such that 
	\begin{center}\label{diag:relative-class-grp}
		\begin{tikzcd}	
			0\arrow{r} & \Cl_{L/F} \arrow{r}\arrow[d] & \Cl_{L} \arrow[r,"\Nm_{L/F}"]\arrow[d,"\Nm_{L/k}"] & N \arrow{r}\arrow[d] & 0 \\		
			0\arrow{r} & 0 \arrow{r} & \Cl_k \arrow{r} & \Cl_k \arrow{r} & 0 \\
		\end{tikzcd}
	\end{center}
	Taking the kernel of the two short exact sequences, we get
	\begin{center}
		\begin{tikzcd}	
			0\arrow{r} & \Cl_{L/F} \arrow{r} & \Cl_{L/k} \arrow{r}& \Cl_{F/k}\cap \Nm_{L/F}(\Cl_L) \arrow{r}& 0.
		\end{tikzcd}
	\end{center}
	Since tensor product with $\mathbb{Z}/\ell\mathbb{Z}$ is right exact, we have for arbitrary integer $\ell$ that
	\begin{center}
		\begin{tikzcd}
			\Cl_{L/F}[\ell]\arrow[r] & \Cl_{L/k}[\ell] \arrow[r,"\Nm"] & (\Cl_{F/k}\cap N) [\ell] \arrow{r} & 0.	
		\end{tikzcd}
	\end{center}
	This proves the first inequality. The second inequality is trivial. The third inequality comes from the fact that $\Coker(\Nm) = \Cl_F/N = \Gal(M/F)$ where $M = h_F\cap L$ is the maximal abelian unramified extension of $F$ inside $L$. The last inequality comes from a combination of an absolute lower bound  $\frac{\Rg_L}{\Rg_k} \ge O_{[L:\Q]}(1)$ by \cite{FrSko} and the theorem of Brauer-Siegel, see for example in \cite{Lou}.
\end{proof}

We are now ready to give the proof of the Extension Lemma \ref{lem:ind-hyp-2}.
\begin{proof}[Proof of Lemma \ref{lem:ind-hyp-2}]
	Every $\tilde{G}$-extension $L/k$ is realized as an $H$-extension $L/K$ over a $G$-extension $F/k$. Given $\mathcal{C}\subset G$, we denote $r = r(\mathcal{C})$ to be the order of elements $c\in \mathcal{C}$. Firstly, we show that if a prime $p$ in $k$ is unramified in $L/k$ and  $\Frob_p(F/k)\in \mathcal{C}\subset G$, then every prime $\fp|p$ above $p$ in $F$ will split in $L/F$. We fix a prime $\mathfrak{P}|p$ in $L/k$ and $\fp = \mathfrak{P}\cap O_F$. Suppose the decomposition group is $D_{\mathfrak{P}/p}= \langle g \rangle \subset \tilde{G}$, then $D_{\fp/p} = \langle g \rangle H/H = \langle \pi(g)\rangle \subset G$ and $D_{\mathfrak{P}/\fp}=\langle g \rangle \cap H\subset H$. Notice that
	\begin{equation}
	D_{\mathfrak{P}/\fp} = e \iff \langle g \rangle \cap H = e \iff g^r=e.
	\end{equation}
	It follows from the assumption on $\mathcal{C}$ that $D_{\mathfrak{P}/\fp} = e$. Therefore every $\fp$ above $p$ will split in $L/F$.
	
	Next, we separate the discussion into two cases based on how large
	$$\eta(L/k):= \frac{\ln \Disc(F/k)}{\ln \Disc(L/k)},$$
	is, i.e., whether $\eta(L/k)\le \eta_0$ or $\eta(L/k)\ge\eta_0$ where 
	$$\eta_0:=\frac{\Delta(\ell, m)}{m\cdot \Delta(\ell,m) + r\cdot\max\{ \beta, \gamma \}}, $$ 
	is the cut-off, and $\beta=\beta(G,k)$, $\gamma=\gamma(G,k)$ and $D_0= D_0(k)$ are parameters in Lemma \ref{lem:MZ}. For the rest of the proof, we will write $\eta$ for $\eta(L/k)$ in short.
	
	\textbf{Case $1$ (Big $\eta$): } If $\eta(L/k)\ge \eta_0$, then we always have
		\begin{equation}\label{eqn:extension-bigF}
		\begin{aligned}
		|\Cl_{L/k}[\ell]|  \le & |\Cl_{F/k}[\ell]| \cdot |\Cl_{L/F}[\ell]| = O_{\epsilon,k}(\Disc(F/k)^{1/2-\delta+\epsilon}) \cdot \frac{\Disc(L)^{1/2+\epsilon}}{\Disc(F)^{1/2+\epsilon}}. \\
		= & O_{\epsilon,k}\Big(\frac{\Disc(L/k)^{1/2+\epsilon}}{ \Disc(F/k)^{\delta}}\Big) = O_{\epsilon,k}( \Disc(L/k)^{1/2-\delta_b(\eta, \ell)+\epsilon}),
		\end{aligned}
		\end{equation}
		where $\delta_b(\eta,\ell) = \delta\cdot \eta$. Here the first inequality follows from Lemma \ref{lem:trivial-bound-relative}, the first equality follows from Lemma \ref{lem:trivial-bound-relative} and the assumption on that $G$ is in $\mathcal{G}(\ell)$ with respect to $\delta = \delta_k(G, \ell)$. The second equality comes from conductor-discriminant formula for relative extensions (notice that we have suppressed the dependence on $k$). The last equality follows from definition of $\eta$. We remark here that we actually do not use the assumption $\eta\ge \eta_0$ here. This bound holds universally true no matter how large $\eta$ is, but it will behave better when $\eta$ is relative large, and when $\eta\ge \eta_0$, we get a uniform saving $\delta_b(\eta, \ell) \ge \delta\cdot \eta_0$ that is independent of $L/k$. So we will need treat the case when $\eta$ is small in another way. 
		
	    \textbf{Case $2$ (Small $\eta$): } If $\eta(L/k)\le \eta_0$, then we separate the discussion when $\Disc(F/k) \le D_0$ and $\Disc(F/k) \ge D_0$ where $D_0 = D_0(k)$ in Lemma \ref{lem:MZ}. 
	    
	    If $\Disc(F/k) \le D_0$, we denote $x = \Disc(L/F)^{\Delta(\ell, m)/r}$. By the standard effective Chebotarev density theorem \cite{LO75}, when $x\ge C_0(k, n)= \exp(10[F:\Q] (\ln D_0 \Disc(k)^{[F:k]})) \ge \exp(10[L:\Q] (\ln D_F)) $, we have
	    \begin{equation}\label{eqn:Chebo-D0}
	    \pi(x;F/k, \mathcal{C}) \ge \frac{1}{2}\frac{|\mathcal{C}|}{|G|} \cdot \frac{x}{\ln x}.
	    \end{equation}
	    Therefore when $\Disc(L/k) \ge C_0(k,n) D_0^m$ is sufficiently large comparing to $k$, we have $\Disc(L/F) = \Disc(L/k) \Disc(F/k)^{-m} \ge \Disc(L/k) D_0^{-m} \ge C_0(k,n)$, thus (\ref{eqn:Chebo-D0}) holds. 
	    
	    If $\Disc(F/k)\ge D_0$, then we apply Lemma \ref{lem:MZ} to $F/k$ with $x = \Disc(L/F)^{\Delta(\ell, m)/r}$, and we obtain
	    \begin{equation}\label{eqn:Chebo-inc}
	   \pi(x;F/k,\mathcal{C}) \ge C_k \frac{1}{\Disc(F/k)^{\gamma}} \cdot \frac{|\mathcal{C}|}{|G|} \cdot \frac{x}{\ln x},
	    \end{equation}
	    when $x\ge \Disc(F/k)^{\beta}$. By the definition of $\eta$, we have
	    $$\Disc(L/F)^{\Delta(\ell, m)/r} \ge \Disc(F/k)^{\beta} \iff \eta \le \eta_0,$$
	    so (\ref{eqn:Chebo-inc}) always holds as long as $\Disc(F/k) \ge D_0$ and $\eta \le \eta_0$. Denote $C'_k= \min\{ 1/2, C_k \}$, then for every $L/k$ with $\eta\le \eta_0$ and $\Disc(L/k)$ sufficiently large, we have
	    	$$\pi(x;F/k,\mathcal{C}) \ge C'_k \frac{1}{\Disc(F/k)^{\gamma}} \cdot \frac{|\mathcal{C}|}{|G|} \cdot \frac{x}{\ln x},$$
	    	which is a lower bound on the number of prime ideals $p$ in $k$ such that $p$ become unramified with $\Frob_p(F/k) \in \mathcal{C} \subset G$. By the argument at the beginning of this proof, all primes $\fp|p$ above $p$ in $F$ will split in $L/F$. Since the inertia degree at $p$ for $F/k$ is $r$, we have $\Nm_{F/\Q}(\fp) = \Nm_{k/\Q}(p)^r$. Therefore
	    	$$\pi(\Disc(L/F)^{\Delta(\ell,m)}; L/F, e) \ge \pi(x;F/k, \mathcal{C}) \ge C'_k \frac{|\mathcal{C}|}{|G|} \cdot \frac{1}{\Disc(F/k)^{\gamma}}  \cdot\frac{\Disc(L/F)^{\Delta(\ell,m)/r } }{\ln \Disc(L/F)^{\Delta(\ell,m)/r }}.$$
	    	Therefore by Lemma \ref{eqn:EV-lemma} we have
	    	\begin{equation}\label{eqn:extension-smallF}
	    	\begin{aligned}
	    	|\Cl_{L/k}[\ell]| = &  O_{\epsilon,k,[F:\Q]}\Big(\frac{\Disc(L)^{1/2+\epsilon}}{ \Disc(L/F)^{\Delta(\ell,m)/r}\cdot \Disc(F/k)^{-\gamma}}\Big)= O_{\epsilon,k}(\Disc(L/k)^{1/2 -\delta_{s}(\eta, \ell) +\epsilon}),\\
	    	\end{aligned}
	    	\end{equation}
	    	where $$\delta_{s}(\eta,\ell) = (1-m\eta)\cdot\Delta(\ell,m)/r- \eta\cdot \gamma.$$
	    	The last equality in (\ref{eqn:extension-smallF}) comes from the definition of $\eta$ and that $\Disc(L/F) = \Disc(L/k)^{1-m\eta}$.
	    
	    Finally, after the discussion for two ranges of $\eta$, we notice that $\delta_{s}(\eta, \ell)$ decreases as $\eta$ increases, and $\delta_b(\eta,\ell)$ increases as $\eta$ increases. It suffices to compare their value at $\eta = \eta_0$:
	    	\begin{equation}
	    	\begin{aligned}
	    	\delta_s(\eta, \ell)= (1-m\eta_0)\cdot\frac{\Delta(\ell,m)}{r}- \eta\cdot \gamma &\ge \delta\cdot \eta_0 = \delta_b(\eta, \ell) \iff\\
	    	\max\{ \beta, \gamma \} - \gamma &\ge \delta. \\
	    	\end{aligned}
	    	\end{equation}
	    	Assuming $\beta>\gamma+1/2$, we will always have $\max\{\beta,\gamma\} - \gamma \ge \delta$. Therefore we can take 
	    	$$\delta_k(\tilde{G},\ell) = \delta_b(\eta_0, \ell) = \delta \cdot \eta_0.$$
\end{proof}

\subsection{Induction by Compositum}\label{ssec:induction-compositum}
In this section, we will prove a lemma on applying the method of Ellenberg-Venkatesh to compositum of number fields. 
\begin{lemma}[Compositum Lemma]\label{lem:compositum}
	Given two permutation groups $G_1\subset S_n$ and $G_2\subset S_m$ and any integer $\ell>1$. Suppose $G_1$ and $G_2$ are both in $\mathcal{G}(\ell)$ with respect to $\delta_i= \delta_k(G_i, \ell)$. Denote $G=G_1\times G_2\subset S_{mn}$ to be a direct product of $G_1$ and $G_2$ as permutation groups, we have $ G \in \mathcal{G}(\ell)$ with respect to $$\delta_k(G,\ell) = \frac{\delta_1\delta_2}{m\delta_2+n\delta_1}.$$
\end{lemma}
\begin{proof}
	Every $G$-extension $L/k$ is the compositum $L_1L_2/k$ of $L_i/k$ where $\Gal(L_i/k) = G_i$ for $i=1,2$ and $\tilde{L}_1/k \cap \tilde{L}_2/k = k$. 
	
	We separate the discussion by
	$$\eta_i(L/k):= \frac{\ln \Disc(L_i/k)}{ \ln \Disc(L/k)}, \quad i=1,2.$$
	It follows from the definition that $\Disc(L_i/k)= \Disc(L/k)^{\eta_i}$.
	
	For any $G$-extension $L/k$, we have
	\begin{equation}\label{eqn:direct-product-eta1}
	\begin{aligned}
	|\Cl_{L/k}[\ell]| \le & |\Cl_{L_1/k}[\ell]| \cdot |\Cl_{L/L_1}[\ell]| = O_{\epsilon,k}(\Disc(L_1/k)^{1/2-\delta_1+\epsilon}) \cdot \frac{\Disc(L)^{1/2+\epsilon}}{\Disc(L_1)^{1/2+\epsilon}}. \\
	= & O_{\epsilon,k}\Big(\frac{\Disc(L/k)^{1/2+\epsilon}}{ \Disc(L_1/k)^{\delta_1}}\Big) = O_{\epsilon,k}( \Disc(L/k)^{1/2-\delta_1\cdot \eta_1+\epsilon}).
	\end{aligned}
	\end{equation}
	Here the first inequality comes from Lemma \ref{lem:trivial-bound-relative}. The first equality comes from the assumption $G_1\in \mathcal{G}(\ell)$. The second equality comes from the conductor-discriminant formula. Similarly, 
	\begin{equation}\label{eqn:direct-product-eta2}
	|\Cl_{L/k}[\ell]|= O_{\epsilon,k}( \Disc(L/k)^{1/2-\delta_2\cdot \eta_2+\epsilon}).
	\end{equation}
	It follows from conductor-discriminant formula that $\Disc(L_1/k)^m \Disc(L_2/k)^n \ge \Disc(L/k)$ when $[L_1L_2:k] = [L_1:k][L_2:k]$. So we get
		\begin{equation}\label{eqn:direct-product-inequality}
		\Disc(L/k)^{m\eta_1} \Disc(L/k)^{n\eta_2} \ge \Disc(L/k),
		\end{equation}
	which gives an inequality between $\eta_i$ that
	\begin{equation}\label{eqn:eta1-eta2}
	m\eta_1+ n\eta_2\ge 1.
	\end{equation}
	
	If $\eta_1\ge M$, then by (\ref{eqn:direct-product-eta1}) we have
	$$|\Cl_{L/k}[\ell]| =  O_{\epsilon,k}( \Disc(L/k)^{1/2-\delta_1\cdot M+\epsilon}).$$
	If $\eta_1\le M$, then $\eta_2\ge \frac{1-M\eta_1}{n}$ from (\ref{eqn:eta1-eta2}). Therefore by (\ref{eqn:direct-product-eta2}) we have
	$$|\Cl_{L/k}[\ell]| = O_{\epsilon,k}( \Disc(L/k)^{1/2-\delta_2 \frac{1-Mm}{n}+\epsilon}).$$
	
	To get the optimal bound for $|\Cl_{L/k}[\ell]|$, we choose $M$ such that
	$$\delta_1\cdot M= \delta_2 \cdot \frac{1-Mm}{n}.$$
	We solve that
	$$M_0 = \frac{\delta_2}{n\delta_1+m\delta_2},$$
	with the corresponding optimal saving $\delta_k(G,\ell)$ is
	\begin{equation}
	\begin{aligned}
	\delta_k(G,\ell) = \delta_1 M_0 = \frac{\delta_1\delta_2}{m\delta_2+n\delta_1}.
	\end{aligned}
	\end{equation}	
\end{proof}
\begin{remark}
	Notice that both (\ref{eqn:direct-product-eta1}), (\ref{eqn:direct-product-eta2}) and (\ref{eqn:direct-product-inequality}) still hold when $L_1/k$ and $L_2/k$ are linearly disjoint, i.e., $[L_1L_2:k] = [L_1:k][L_2:k]$. So the exact same argument of Lemma \ref{lem:compositum} applies with no change to all permutation groups arising from linearly disjoint compositum of a $G_1$ extension with a $G_2$ extension. Equivalently, these are the permutation groups $G\subset G_1\times G_2 \subset S_{mn}$ that is transitive and $S_1:=\{ g_1\in G_1 \mid \exists g_2\in G_2, (g_1, g_2)\in G  \} = G_1$ and similarly $S_2 = G_2$. 
\end{remark}
 
\section{Forcing Sequence for $p$-Groups}\label{sec:group-theory}
In \cite{JW20}, the author has proved Theorem \ref{thm:main} for $G = (\zp)^r$ with $r>1$. Notice that the Frattini quotient of a $r$-generated $p$-group is always isomorphic to $(\zp)^r$. This leads to our strategy to prove Theorem \ref{thm:main} for $r$-generated $p$-groups, i.e., to do an induction via Extension Lemma \ref{lem:ind-hyp-2} with the base case $G=(\zp)^r$. 

The key group theoretic lemma we will prove is the following. It is a crucial input for applying Extension Lemma \ref{lem:ind-hyp-2}.
\begin{theorem}\label{thm:p-grp-forcing}
	Every non-cyclic and non-quaternion $p$-group $G$ has a decreasing sequence of normal subgroups $N_i$ 
	$$ G \supset \Phi(G)= N_0 \supset N_1\supset N_2\supset \cdots \supset N_m = {e},$$
	where for every $0\le i<m$: \\
	1) $[N_i: N_{i+1}] = p$;\\
	2) $(G/N_{i+1}, \pi)$ is a forcing extension of $G/N_{i}$ where $\pi: G/N_{i+1} \to G/N_{i}$.
\end{theorem}
See the proof in section \ref{ssec:p-grp-odd} and section \ref{sssec:proof-2-grp}. In general we will say a sequence $G= N_0\supset N_1\supset \cdots \supset N_m=e$ of normal subgroups of $G$ is a \emph{forcing sequence of $G$} if $(G/N_{i+1}, \pi)$ is a forcing extensions of $G/N_i$ where $\pi: G/N_{i+1}\to G/N_i$ for every $0\le i<m$.

\subsection{Basics for $p$-Group}
We first introduce some basic concepts for $p$-groups. 
%It is a standard fact that nilpotent group is isomorphic to the direct product of its Sylow-$p$ subgroups $G_p$ over all $p$, see for example Theorem $1.26$ \cite{Isaacs}. Therefore to understand nilpotent groups, for most applications, it suffices to consider $p$-groups first.
Given a finite $p$-group $G$, the \textit{Frattini subgroup} $\Phi(G)\subset G$ is defined to the intersection of all maximal subgroups of $G$. 
%Equivalently, it is the subgroup of all non-generating elements of $G$ where $x$ is \emph{non-generating} if any subset $S\cup \{ x\}$ of $G$ generates $G$ implies that $S$ generates $G$. 
%\jiuya{This equivalence can be seen in the following way: For one direction: if $x$ is not non-generating, then there exists $T$ such that $T\cup \{ x\}$ generate $G$, but $T$ itself does not, this means that $T$ generate a non-trivial subgroup $H$ of $G$, and it means that $G/ H(T)$ is generated by $x$, which must be a cyclic $p$-group, denotes the subgroup generated by $T\cup \{ x^p \}$, we have the corresponding quotient a cyclic order $p$ group generated by $x$, so $x$ is not an element of this maximal subgroup which is generated by $T\cup \{x^p \}$. For the other direction, if $x$ is non-generating, we want to show that $x$ is contained in every maximal subgroup of $G$. Suppose there exists one subgroup that is not maximal and $x\notin H$, then $H\cup \{x \}$ generate $G$, but $H$ does not, so it contradicts with being non-generating.}
We call $G/\Phi(G)$ the \textit{Frattini quotient} of $G$. It follows from Burnside's basis theorem that $G/\Phi(G)$ is the largest elementary abelian quotient of $G$. Therefore $\Phi(G) = G^p\cdot [G,G]$ is the subgroup generated by the set of all the $p$-th power $G^p$ and the commutator $[G, G]$. It is clearly normal since it is a characteristic subgroup. Now suppose $G/\Phi(G) \simeq (\zp)^r$, then we say the \textit{generator rank} of $G$ is $r$. Equivalently, the generator rank is $r = \dim(H^1(G, \zp))$ where $\zp$ is considered as a trivial $G$-module. 

We define $G_0 = G$ to be the group itself and $G_1 = \Phi(G)$ to be the Frattini subgroup. Inductively we define $G_j:= G_{j-1}^p \cdot [G_{j-1}, G]$ for $j>0$, equivalently $G_j$ is defined to be the minimal subgroup such that $G_{j-1}/G_j$ is central in $G/G_j$ with exponent $p$. These subgroups form a strictly decreasing sequence of subgroups
$$G = G_0  \supset G_1\supset G_2\supset \cdots \supset G_c = \{ e\}.$$
This sequence is called \emph{lower exponent $p$ central series} of $G$. We define the minimal integer $c$ such that $G_c = \{ e\}$ to be \textit{$p$-class} of the finite $p$-group $G$. We will denote the $p$-class of $G$ by $c(G)$. We will write $\bar{G}_k := G/G_k$ in short.

We can parametrize all finite $p$-groups with generator rank $r$ by the \textit{$p$-group generating algorithm} \cite{OBrien} by putting all $p$-groups into a \emph{descendant tree}. The root of the tree is the elementary abelian $p$-group $A =(\zp)^r$. The \textit{immediate descendants} of a finite $p$-group $G$ are all finite $p$-groups $D$ such that $D/D_{c-1} \simeq G$ where $c = c(D)$. Since these characteristic groups $G_j$ are defined inductively by an explicit formula, one can show that the operation of taking $j$-th subgroup in the sequence commutes with group homomorphism, i.e., if $f: M\to N$ are two $p$-groups, then $f(M_j) = N_j$ for all $j>0$. Therefore if $D/D_k\simeq G$ for some $1<k<c(D)$, then $c(G)=k$, and $D/D_j \simeq G/G_j$ for all $0<j<c(G)$. This guarantees that ancestors of a $p$-group $G$ are the quotients $G/G_j$ with $0< j< c(G)$, the descendants of a finite $p$-group $G$ are all finite $p$-groups $D$ such that $D/D_k\simeq G$ where $0<k<c$ and $c$ is the $p$-class of $D$. Since $G/G_1$ is always isomorphic to one elementary abelian group, $G$ belongs to the unique tree with the root $(\zp)^r$ where $r = r(G)$. In particular, a $p$-group $G$ is one descendant in $c(G)$-th generation if we count elementary abelian group as the $1$-st generation. This tree encodes many properties of $p$-groups. If a $p$-group $G$ does not have any descendants, equivalently there are no $p$-group $D$ with $D/D_j\simeq G$ where $D_j$ defined in the sequence, then we call such a group a \textit{leaf}. 

\subsection{Odd $p$-Group}\label{ssec:p-grp-odd}
In this section, we will prove Theorem \ref{thm:p-grp-forcing} for odd $p$-groups. Notice that for odd $p$, there is no quaternion group, so we will prove Theorem \ref{thm:p-grp-forcing} for every non-cyclic odd $p$-groups. 

\begin{lemma}\label{lem:p-grp-series}
	Given a $p$-group $G$, there exists a series of normal subgroups 
	$$G\supset \Phi(G) = N_0\supset N_1\supset \cdots\supset N_m= e,$$
	where for every $0\le i< m$:\\
	1) $[N_{i}:N_{i+1}]= p$;\\
	2) $N_{i}/N_{i+1}$ is in the center of $G/N_{i+1}$;\\
	3) the sequence is a refinement of the lower exponent $p$ central series of $G$, or equivalently, for all $j<c(G)$ the subgroup $G_j$ is equal to $N_i$ for some $i$.
\end{lemma}
\begin{proof}
	For every $p$-group $G$ with $c(G)=c$, let's say $G=G_0\supset G_1\supset \cdots\supset G_c=\{e \}$ is the lower exponent $p$ central series of $G$. By construction, for every $j<c$, the subgroup $G_j$ is normal in $G$ since it is characteristic, $G_{j}/G_{j+1}$ is in the center of $G/G_{j+1}$, and $G_j/G_{j+1}$ has exponent $p$. 
	
	Fix $j$. Let $G_j = S_0\supset S_1\supset \cdots \supset S_K=G_{j+1}$ be an arbitrary refinement of $G_j\supset G_{j+1}$ with $[S_k:S_{k+1}] = p$. Since $G_j/G_{j+1}$ is in the center of $G/G_{j+1}$, for all $k$, we have $S_k/S_{k+1}$ is in the center of $G/S_{k+1}$. We will show that $S_k$ is also normal in $G$ for all $k$. In fact, since $S_k/G_{j+1} \subset G/G_{j+1}$ is in the center of $G/G_{j+1}$, clearly $S_k/G_{j+1}$ is normal in $G/G_{j+1}$. Then denote $\phi: G\to G/G_{j+1}$ to be the canonical projection, the preimage $S_k = \phi^{-1}(S_k/G_{i+1})$ is also normal in $G$. 

%	Notice that every subgroup $S$ of $G_{i-1}/G_i$ is also normal in $G/G_i$ since $S$ is in the center of $G/G_i$, therefore denote $\pi_i: G\to G/G_i$ to be the natural projection, then $\pi_i^{-1}(S)\subset G$ is also normal in $G$. So let's take the series $G_{i-1} = G_{i-1,0} \supset G_{i-1,1}\supset \cdots \supset G_{i-1, r_{i-1}}  = G_i$ of subgroups between $G_{i-1}$ and $G_i$ such that $[G_{i,j}: G_{i,j+1}]=p$ where $[G_{i-1}:G_i] = p^{r_{i-1}}$. The choice for the intermediate subgroups $G_{i,j}$ is not canonical, but we can see that for each $j$, we can show that the sub-quotient $G_{i,j}/G_{i,j+1}$ is central in $G/G_{i,j+1}$. Indeed it follows from the definition of $G_i$ that $G_{i,j}/G_i$ is central in $G/G_i$, so $G_{i,j}/G_{i,j+1}$ is central in $G/G_{i,j+1}$. 	
	Therefore we could refine the lower exponent $p$ central series by inserting normal subgroups $S_{k,j}$ as above between $G_j$ and $G_{j+1}$ for every $j>0$. We will get a descending sequence of subgroups $G\supset \Phi(G) = N_0\supset N_1\supset \cdots\supset N_m= 1$ that satisfies all three conditions.  
\end{proof}

\begin{proof}[Proof of Theorem \ref{thm:p-grp-forcing} for odd $p$]
Let $G$ be a non-cyclic odd $p$-group. By Lemma \ref{lem:p-grp-series}, we have a sequence
$$G\supset \Phi(G) = N_0\supset N_1\supset \cdots\supset N_m= e,$$
of descending subgroups, where $N_i/N_{i+1} \simeq \zp$ in the center of $G/N_{i+1}$. It suffices to prove that the following extension $(G/N_{i+1}, \pi)$ is forcing for each $i$.
	\begin{center}
		\begin{tikzcd}
			0\arrow{r} & N_i/N_{i+1} \arrow[r] & G/N_{i+1} \arrow[r,"\pi_i"] & G/N_{i} \arrow{r} & 0.
		\end{tikzcd}
	\end{center}

Since $N_i/N_{i+1}$ is in the center of $G/N_{i+1}$, the extension $\pi$ is a central extension. Suppose $g_0\in G/N_i$ is not identity. Denote the conjugacy class containing $g_0$ by $\mathcal{C}\subset G/N_i$. We will show that for any $c\in \mathcal{C}$, all elements in $\pi^{-1}(c)$ have the same order. If $\pi(\tilde{c}) = c$, then all preimages of $c$ is $a\tilde{c}$ for $a\in N_i/N_{i+1}$. Since $\ord(a)=p$ and $a$ is central we get $\ord(\tilde{c}) = \ord(a\tilde{c})$. On the other hand, if $c' = x^{-1}cx\in G/N_i$, then $\tilde{x}^{-1} \tilde{c} \tilde{x}\in \pi^{-1}(c')$ when $\tilde{x}\in \pi^{-1}(x)$. It follows that $\ord(\tilde{x}^{-1} \tilde{c} \tilde{x} ) = \ord(\tilde{c})$.

Therefore it suffices to find an element $y\in G/N_{i+1}$ such that $\ord(y) = \ord(\pi(y))$. If $G$ is not cyclic, then $G/N_{i+1}$ is not cyclic since $G/\Phi(G) = G/N_{i+1}/ \Phi(G/N_{i+1}) = (\zp)^r$ with $r>1$, then when $p$ is odd there exists at least two cyclic subgroups of order $p$, see e.g. Theorem $12.5.2$ in \cite{Hall}. Therefore there must be a subgroup $T$ of order $p$ and $T\neq N_i/N_{i+1}$. Denote the generator of $T$ by $y$. By construction, $\ord(y)=p = \pi(y)$ since $y\notin N_i/N_{i+1}$. Then $\pi$ is forcing with respect to the conjugacy class $\mathcal{C}\subset G/N_i$ of $\pi(y)$.
\end{proof}

\subsection{Even $p$-Group}
In this section, we will prove Theorem \ref{thm:p-grp-forcing} for $2$-groups. Such a good picture for odd $p$-groups where all non-cyclic $p$-groups satisfy Theorem \ref{thm:p-grp-forcing} no longer holds for $2$-group. We will first introduce these exceptional groups, \emph{generalized quaternion groups}, and list their properties in section \ref{sssec:quaternion}. Then we will give the proof for all non-cyclic and non-quaternion $2$-groups in section \ref{sssec:proof-2-grp}.

\subsubsection{Generalized Quaternion Groups}\label{sssec:quaternion}
We define the \textit{generalized quaternion group} by 
$$Q(n):=\large\langle x, y\mid x^{2^{n+1}} = y^4 = 1, x^{2^{n}}= y^2, y^{-1}xy= x^{-1} \large \rangle.$$
When $n=1$, we get the smallest such group, which is usually called quaternion group and denoted by $Q_8$. The generalized quaternion groups have the special property that all abelian subgroups are cyclic, see Figure \ref{diag:quaternion} for the subgroup lattice of $Q(1)$ as an example.

We will prove that this family of $2$-groups is the only exceptional groups aside from cyclic $2$-groups for Theorem \ref{thm:p-grp-forcing} in section \ref{sssec:proof-2-grp}. In preparation for the proof, we will first list several useful properties of $Q(n)$ in Lemma \ref{lem:gen-quarternion-property}.

Before we state the properties, we briefly recall the concept of \emph{Schur multiplier}. Given a finite group $G$, we say $E$ is a \textit{stem extension} of $G$ 
\begin{center}
	\begin{tikzcd}
		0\arrow{r} & Z\arrow[r] & E \arrow{r} & G \arrow{r} & 0,
	\end{tikzcd}
\end{center}
if $Z\subset [E,E] \cap Z(E)$ where $Z(E)$ is the center of $E$. We then define \textit{Schur multiplier} $M(G)$ of $G$ to be the kernel of the unique largest stem extension of $G$. Equivalently, if $G  = F/R$ where $F$ is a free group, then there is a formula $M(G) = R\cap [F, F]/ [R, F]$ for Schur multiplier. 

\begin{lemma}\label{lem:gen-quarternion-property}
	The generalized quaternion group $Q(n)$ has the following property:
	\begin{enumerate}
		\item The order of $Q(n)$ is $2^{n+2}$. 
		\item The $2$-class of $Q(n)$ is $n+1$.
		\item The center of $Q(n)$ is $\mathbb{Z}/2\mathbb{Z}$.
		\item It has trivial Schur multiplier.
		\item It is a leaf on the descendant tree.
	\end{enumerate}
\end{lemma}
\begin{proof}
	\begin{enumerate}
		\item 
		Consider the cyclic subgroup $N=\langle x \rangle$ generated by $x$. Then $Q(n)/N= C_2$ since $y^2\in N$. Therefore $|Q(n)| = 2^{n+2}$.
		
		\item
 	    We can write down the exponent $p$ lower central series for $Q(n)$.	Recall that $G_1=\Phi(G) = G^2[G,G]$. By definition, it is clear that $x^2\in G_1$, and $G_1/\langle x^2 \rangle = C_2\times C_2 = \langle \bar{x}, \bar{y} \rangle$. So $G_1 = \large \langle x^2 \large\rangle$ is a cyclic group with order $2^{n}$. For $k=2$, notice that the only subgroup $G_2$ of $G_1$ with $G_1/G_2$ exponent $2$ is
		$G_2 =  \large \langle x^{4} \large\rangle.$ Similarly $G_k =  \large \langle x^{2^k} \large\rangle$. Therefore we have the $2$-class of $Q(n)$ is $n+1$. 
		
		\item
		Suppose $x^s$ is in the center, then $y\cdot x^s = x^s \cdot y = y\cdot x^{-s}$ implies that $s = 2^n$. One can show that it is the only element that commute both with $x$ and $y$. Therefore $Z(Q(n)) = \{ e, x^{2^n} \}$.

		\item
		See Exercise $5A.7$ in \cite{Isaacs}.
		
		\item
		Given $G = Q(n)$, we have shown that $G_k = \langle x^{2^k} \rangle$, and $(G/G_k)^{ab} = G^{ab} = C_2\times C_2$ for every $k>0$. The abelianization $G^{ab}\simeq (G/G_{c-1})^{ab}$ where $c = c(G)$, then all immediate descendants $D$ of $G$ must have $D^{ab} = G^{ab}$ by Theorem $4.4$ \cite{Nover}. Therefore if $G$ has any immediate descendant $D$, then $D$ is a central extension of $G$ with $D^{ab} = G^{ab}$. By definition, a central extension is a stem extension if and only if $E^{ab} = G^{ab}$. So the existence of immediate descendants contradicts with $G$ having trivial Schur multiplier. 
	\end{enumerate}
\end{proof}

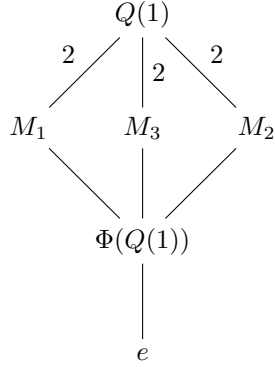
\begin{figure}
	\centering
	\begin{tikzpicture}[node distance = 1.5cm, auto, bend angle=45]
	\node (Gi)  {$Q(1)$};
	\node (Gi1a) [below of=Gi, left of=Gi] {$M_1$};
	\node (Gi1b) [below of=Gi, right of=Gi] {$M_2$};
	\node (Gi1c) [below of=Gi] {$M_3$};
	\node (Gi+) [below of=Gi1a, right of=Gi1a] {$\Phi(Q(1))$};
	\node (Gi++) [below of=Gi+] {$e$};
	
	\draw[-] (Gi) to node [swap] {$2$}  (Gi1a);
	\draw[-] (Gi) to node  {$2$}  (Gi1b);
	\draw[-] (Gi) to node  {$2$}  (Gi1c);
	\draw[ -] (Gi1a) to node  {}  (Gi+);
	\draw[-] (Gi1b) to node  {}  (Gi+);
	\draw[ -] (Gi1c) to node  {}  (Gi+);
	\draw[-] (Gi+) to node  {}  (Gi++);
	\end{tikzpicture}
	\caption{Quaternion Group of Order $8$} \label{diag:quaternion}
\end{figure}

We can see that Theorem \ref{thm:p-grp-forcing} cannot hold for generalized quaternion group. As an example, the smallest quaternion group $Q_8$ is a central extension of $\mathbb{Z}/2\mathbb{Z}\times \mathbb{Z}/2\mathbb{Z}$,
\begin{center}
	\begin{tikzcd}
		0\arrow{r} & \mathbb{Z}/2\mathbb{Z} \arrow[r] & Q_8 \arrow[r,"\pi"] & \mathbb{Z}/2\mathbb{Z}\times \mathbb{Z}/2\mathbb{Z} \arrow{r} & 0.
	\end{tikzcd}
\end{center}
Such an extension $(Q_8, \pi)$ is not forcing since for every element $c\in \mathbb{Z}/2\mathbb{Z}\times \mathbb{Z}/2\mathbb{Z}$, elements in $\pi^{-1}(c)$ all have order $4$ whereas $c$ has order $2$. One can similarly show that for general $n$, the extension $(Q(n), \pi)$ where $\pi: Q(n) \to Q(n)/Z(Q(n))$ is not forcing. This failure has a lot to do with the fact that $Z(Q(n))$ is the only $\mathbb{Z}/2\mathbb{Z}$ subgroup of $Q(n)$. This turns out to be a characterizing property of $Q(n)$ by the following lemma.

\begin{lemma}[Theorem $12.5.2$, \cite{Hall}]\label{lem:Hall}
	A $p$-group which contains only one subgroup of order $p$ is cyclic or generalized quaternion group. 
\end{lemma}

\subsubsection{Proof of Theorem \ref{thm:p-grp-forcing} for $p=2$}\label{sssec:proof-2-grp}
In last section, we have shown that Theorem \ref{thm:p-grp-forcing} does not hold for quaternion group and cyclic group. Therefore the best we can hope for is that Theorem \ref{thm:p-grp-forcing} is true for all $2$-groups that are non-quaternion and non-cyclic. We will show that is really the case!

Unlike the case for odd $p$ where an arbitrary refinement of the lower exponent $p$ central series satisfies the property stated in Theorem \ref{thm:p-grp-forcing}, when $p=2$, it can happen that some refinement of the lower exponent $p$ central series will not be forcing when the refinement $G/N_i \simeq Q(n)$ for some $i$. Therefore our main focus in the following proof is to show that we can always find a detour in the refinement to avoid such quaternion quotients. 

\begin{proof}[Proof of Theorem \ref{thm:p-grp-forcing} for $p=2$]
	We will separate the discussion for $2$-group $G$ with generator rank $r=2$ and $r>2$.
	
    Firstly, we consider the case when $G$ is a $2$-group with $r>2$. Then $G$ is non-cyclic and non-quaternion since cyclic $2$-group has $r=1$ and quaternion group has $r=2$. By Lemma \ref{lem:p-grp-series}, we have a sequence $$G\supset \Phi(G) = N_0\supset N_1\supset \cdots\supset N_m= e,$$
    of descending groups, where $N_i/N_{i+1} \simeq \zp$ in the center of $G/N_{i+1}$. We will to show that the following extension is forcing for every $i$,
    	\begin{center}
    		\begin{tikzcd}
    			0\arrow{r} & N_i/N_{i+1} \arrow[r] & G/N_{i+1} \arrow[r,"\pi_i"] & G/N_{i} \arrow{r} & 0.
    		\end{tikzcd}
    	\end{center}
    	Since $G/\Phi(G) = (\zp)^r$ and $Q(n)$ has $r=2$, the quotient $G/N_{i+1}$ is not quaternion for every $i$. Then by Lemma \ref{lem:Hall}, there must exist a subgroup $T = \langle y \rangle$ of order $2$ and $T\neq N_i/N_{i+1}$. Then the same proof for odd $p$-group shows that $\pi$ is forcing with respect to the conjugacy class $\mathcal{C}\subset G/N_i$ of $\pi(y)$.
    	
    	Secondly, we consider the case when $G$ has $r=2$ and is non-quaternion. It suffices to construct a sequence 
    	$$G\supset \Phi(G) = N_0\supset N_1\supset \cdots\supset N_m= e,$$
    	where for every $i$, $[N_i:N_{i+1}]=2$, $N_i/N_{i+1}$ is in the center of $G/N_{i+1}$, and finally $G/N_i$ is non-quaternion. Indeed if we find such a sequence then the proof for $r>2$ carries over. 
    	
    	We firstly take the lower exponent $p$ central series $G=G_0\supset \cdots \supset G_j \supset \cdots \supset G_c = 1$ where $c$ is the $2$-class of $G$. By the construction of $G_j$, the quotient $G/G_j$ has $2$-class $c(G/G_j)=j$. By Lemma \ref{lem:gen-quarternion-property}, the group $Q(n)$ has no descendants, so $G/G_j$ is non-quaternion for every $j<c$. For $j=c$, $G/G_c = G$ is non-quaternion by assumption. Then we start to refine the exponent $p$ lower central series of $G$. We denote the dimension of $[G_j:G_{j+1}]$ to be $r_j$, i.e., $[G_j:G_{j+1}] = 2^{r_j}$. If $r_j>1$ for certain $j$, then we have multiple options to choose intermediate subgroups $G_j = G_{j,0}\supset G_{j,1} \supset \cdots G_{j,i} \cdots \supset G_{j, r_j}= G_{j+1}$ for $0\le i\le r_j$ with $[G_{j,i}:G_{j,i+1}]=2$. By the proof of Lemma \ref{lem:p-grp-series}, an arbitrary choice of refinement $G= N_0\supset N_1\supset \cdots \supset N_m=e$ we choose will satisfy that $[N_i:N_{i+1}]=p$ and $N_i/N_{i+1} \subset G/N_{i+1}$.
    	
    	 	\begin{figure}
    	 		\centering
    	 		\begin{tikzpicture}[node distance = 1.5cm, auto, bend angle=45]
    	 		\node (G0) {$G=G_0$};
    	 		\node (G1) [below of=G0] {$G_1$};
    	 		\node (G2) [below of=G1] {$G_2$};
    	 		\node (Gi)  [below of=G2] {$G_j$};
    	 		\node (Gi1a) [below of=Gi, left of=Gi] {$G_{j,1}$};
    	 		\node (Gi1b) [below of=Gi, right of=Gi] {$G'_{j,1}$};
    	 		\node (Gi1c) [below of=Gi] {$G''_{j,1}$};
    	 		\node (Gi+) [below of=Gi1a, right of=Gi1a] {$G_{j,2}$};
    	 		\node (Gi++) [below of=Gi+] {$G_{j+1}$};
    	 		
    	 		\draw[-] (G0) to node  {$p^2$}  (G1);
    	 		\draw[-] (G1) to node  {$p$}  (G2);
    	 		\draw[dashed, -] (G2) to node  {}  (Gi);
    	 		\draw[-] (Gi) to node [swap] {$p$}  (Gi1a);
    	 		\draw[-] (Gi) to node  {$p$}  (Gi1b);
    	 		\draw[-] (Gi) to node  {$p$}  (Gi1c);
    	 		\draw[ -] (Gi1a) to node  {}  (Gi+);
    	 		\draw[-] (Gi1b) to node  {}  (Gi+);
    	 		\draw[ -] (Gi1c) to node  {}  (Gi+);
    	 		\draw[dashed, -] (Gi+) to node  {}  (Gi++);
    	 		
    	 		\draw[-, out=200, in=120] (G0) to node [swap] {$Q(j)$}  (Gi1a);
    	 		\end{tikzpicture}
    	 		\caption{Subgroup lattice of $G$} \label{diag:series-quaternion}
    	 	\end{figure}
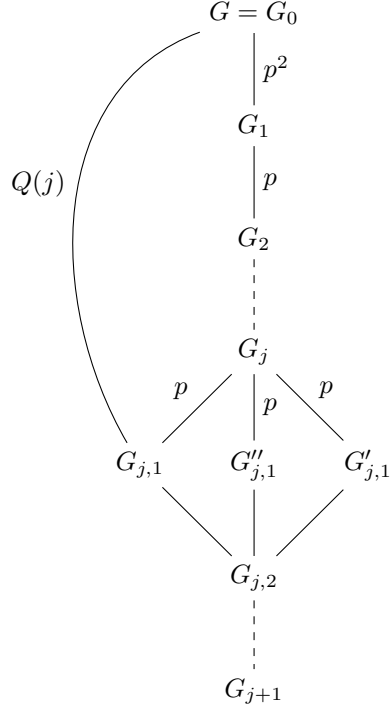
    	
    	We will prove that we can refine the lower exponent $p$ central series in a careful way so that none of the quotient is quaternion. Fix $j$. Firstly, by the standard property of $p$-class and lower exponent $p$ central series, we have that $c(G/G_{j,i})=j+1$ for all $0< i\le r_j$. If $G/G_{j,i} \simeq Q(s)$ for some $s$, then $Q(s)$ has $2$-class $j+1$. By Lemma \ref{lem:gen-quarternion-property}, we must have $s = j$. Again by Lemma \ref{lem:gen-quarternion-property}, we get $|Q(j)| = [G:G_{j,i}] = 2^{j+2}$. However since $G$ has generator rank $2$ and $[G:G_1] = 2^2$, we must have $[G_m: G_{m+1}]=2$ for every $1\le m< j$ and $i=1$. Since $r_j>1$, there are at least $3$ options in choosing $G_{j,1}$, see Figure \ref{diag:series-quaternion}. Suppose $G/G_{j,1} \simeq Q(j)$, then since $G_j/G_{j,1}$ is central in $G/G_{j,1}$ and by Lemma \ref{lem:gen-quarternion-property} the center of $Q(j)$ is cyclic of order $2$, then we can see that $G_j = Q(j)/Z(Q(j))$. Suppose for two of the choices, we get both quotients $G/G_{j,1} \simeq G/G'_{j,1}\simeq Q(j)$ isomorphic to $Q(j)$. Then by the universal property of fibered product, we get
    	$$ G/G_{j,2} = Q(j)\times_{G_j} Q(j) \simeq Q(j)\times C_2.$$
    	However $Q(j)\times C_2$ has generator rank $3$, and $G/G_{j,2}$ is a quotient of $G/G_{j+1}$, therefore must have generator rank at most $2$. Contradiction. So we prove that at most one of the $3$ choices of $G_{j,1}$ satisfies $G/G_{j,1} = Q(j)$, and therefore we can always find a normal subgroup $G_{j,1}$ such that $G/G_{j,1}$ is non-quaternion.   	   
\end{proof}

\section{$\ell$-torsion in Class Group of Nilpotent Extensions}\label{sec:main-proof}
\subsection{Proof of the Main Theorem}
In this section, we will prove the following theorems building on results in section \ref{sec:induction} and \ref{sec:group-theory}.
\begin{theorem}\label{thm:p-grp-main}
	Given an arbitrary integer $\ell>1$ and any number field $k$, the regular representation of a $p$-group $G$ is in $\mathcal{G}(\ell)$ if $G$ is non-cyclic and non-quaternion.
\end{theorem}
\begin{proof}
	By Theorem \ref{thm:p-grp-forcing}, for every non-cyclic and non-quaterion $p$-group $G$, we can find a decreasing sequence of normal subgroups $N_i$ such that the following is a forcing extension
		\begin{center}
			\begin{tikzcd}
				0\arrow{r} & N_i/N_{i+1} \arrow[r] & G/N_{i+1} \arrow[r,"\pi_i"] & G/N_{i} \arrow{r} & 0.
			\end{tikzcd}
		\end{center}
	for each $i$, and $G/N_0 = G/\Phi(G) = (\zp)^r$ is an elementary abelian group with generator rank $r$. 
	
	We will apply induction on $i$. For $i=0$, we have $G/N_0 \in \mathcal{G}(\ell)$ by \cite{JW20}. Suppose that $G/N_i\in \mathcal{G}(\ell)$, then by Extension Lemma \ref{lem:ind-hyp-2}, we have $G/N_{i+1} \in \mathcal{G}(\ell)$ since $(G/N_{i+1}, \pi_i)$ is a forcing extension. 
\end{proof}

\begin{theorem}\label{thm:nil-grp-main}
	Given any integer $\ell>1$ and any number field $k$, the regular representation of a nilpotent group $G$ is in $\mathcal{G}(\ell)$ if for every $p||G|$, the Sylow-$p$ subgroup $G_p\subset G$ is non-cyclic and non-quaternion.
\end{theorem}
\begin{proof}
	It is a standard fact that a nilpotent group $G$ is the direct product of its Sylow-$p$ subgroups $G_p$, i.e., $G = \prod_{p||G|} G_p$. If for every $p||G|$, the subgroup $G_p$ is non-cyclic and non-quaternion, then by Theorem \ref{thm:p-grp-main}, all $G_p\in \mathcal{G}(\ell)$. Using Lemma \ref{lem:compositum} inductively, we get $G\in \mathcal{G}(\ell)$.
\end{proof}

\begin{remark}
	We mention that Theorem \ref{thm:p-grp-main} and \ref{thm:nil-grp-main} will also result in corresponding improvements in upper bounds for Malle's conjecture, discriminant multiplicity conjecture and generalized version for these conjectures in \cite{EV06}, for implications of these conjectures see \cite{EV06,PTBW2}. 
\end{remark}

\subsection{Discussion on $D_4$}
In \cite{Ellen16}, all number fields with degree less or equal to $5$ are shown to have non-trivially bounded $\ell$-torsion in class groups on average, with $D_4$ being the only exceptional case. In \cite{ML17}, the method of using $L$-functions also does not seem to apply to $D_4$ quartic extensions since a positive density of $D_4$ extensions can contain a common subextension. In order to address this issue for $D_4$ extensions, it is suggested in \cite{ML17} and proved in \cite{Chen}, that when one considers the family of $D_4$-extensions containing a common $C_2\times C_2$ quotient $M$, denoted by $\mathcal{F}_M$, the obstacle from the common subfield is avoided. Precisely, the $\ell$-torsion in class groups $|\Cl_F[\ell]|$ is non-trivially bounded on average when $F$ is among the family of all $D_4$-quartic extensions over $\Q$ with a fixed $C_2\times C_2$ quotient (with a pointed $C_2$ quotient).

We remark that our result cannot prove that the permutation group $D_4\subset S_4$ is in $\mathcal{G}(\ell)$ yet, however, the regular representation of $D_4$, as shown by our proof, is in $\mathcal{G}_k(\ell)$ for every number field $k$ and every integer $\ell$. 

This gives an improvement on \cite{Chen}. Let's denote $F$ to be a $D_4$ quartic extension. When we impose the condition on $\tilde{F}$ having a fixed $C_2\times C_2$ quotient $M_0$ along with a pointed $C_2$ quotient $K_0$, there is a fixed quadratic subfield $K_0$ for all $F$. Notice that $|\Cl_F[\ell]| = |\Cl_{K_0}[\ell]| \cdot|\Cl_{F/K_0}[\ell]|$ when $\ell$ is odd and $|\Cl_{F/K_0}[\ell]|^2 = |\Cl_{\tilde{F}/M_0}[\ell]|$. Therefore a non-trivial bound on $\Cl_F[\ell]$ on average within $\mathcal{F}_{M_0}$ is equivalent to a non-trivial bound on $|\Cl_{\tilde{F}}[\ell]| = |\Cl_{M_0}[\ell]| \cdot|\Cl_{\tilde{F}/M_0}[\ell]| $ on average within the family of all $D_4$ octic extensions $\tilde{F}$ with a fixed $C_2\times C_2$ quotient $M_0$. Theorem \ref{thm:p-grp-main} proves that we can actually prove a point-wise non-trivial bound for $\Cl_{\tilde{F}}[\ell]$ for every $D_4$-octic extensions $\tilde{F}$. This means that we not only drop the "on average" condition, moreover, we drop the condition on containing a fixed $C_2\times C_2$ quotient.

\subsection{On $\delta_k(G, \ell)$}
In this section, we give a brief discussion on the amount of power saving $\delta_k(G, \ell)$. 

We remark that there are potentially several sources of optimizing the pointwise saving $\delta_k(G, \ell)$. For example in Theorem \ref{thm:p-grp-main}, notice that for a $p$-group $G_p$, when $p|\ell$, we can always write $\ell  = \ell_p \cdot \ell_{(p)}$ where $\ell_p$ is the maximal $p$-power divisor of $\ell$ and $\ell_{(p)}$ is the maximal divisor relatively prime to $p$. Writing $|\Cl_F[\ell]| = |\Cl_F[\ell_p]| \cdot |\Cl_F[\ell_{(p)}]|$, we can thus use the perfect bound for $\Cl_F[\ell_p]$ and use the method of Theorem \ref{thm:p-grp-main} for the part $\Cl_F[\ell_{(p)}]$. Another source of improving the saving is to construct different forcing sequences for a single $p$-group. 

Although we do not intend to give optimal savings for this work, we will give a quantification on how much saving one can derive away from the trivial bound from this work. Since the expression of $\delta_k(G, \ell)$ in general will be very complicated after applying the induction, we will only give an estimation (actually a lower bound on $\delta_k(G, \ell)$) in the main example: $k = \Q$, $G$ is a $p$-group with $p$ odd and $\ell\neq p$ is another odd prime. 

\begin{example}[$k=\Q$, $p\neq \ell$ both odd and prime]
Let $G$ be a $p$-group with order $p^n$ and generator rank $r$, and $\ell\neq p$ be an odd prime. By \cite{ZamThesis}, we can take $\gamma = 19$ and $\beta = 35$ universally for any $k$ and $G$ in Lemma \ref{lem:MZ}. For $G/\Phi(G)= (\zp)^r$ with $r>1$, by \cite{JW20}, we know that 
$$\delta_0 = \delta_{\Q}(G/\Phi(G), \ell) = \frac{\Delta(\ell,p)}{p(1+t_0)},$$
where $t_0 = 1/(p-1)\Delta(\ell,p)(1-2/p)$. For each step of induction, by Lemma \ref{lem:ind-hyp-2}, the saving $\delta_{\Q}(G, \ell)$ gets an extra factor $\eta_0$. Notice that by taking the forcing sequence for $G$ constructed by Theorem \ref{thm:p-grp-forcing}, we always have $r= p$ and $m = p$. Therefore we know that
$$\delta_{\Q}(G, \ell) = \delta_0 \cdot \eta_0^{n-r},$$
where $\eta_0 = \frac{1}{p} \frac{\Delta(\ell, p)}{\cdot \Delta(\ell, p)+ 35} \ge \frac{1}{72 p^2 \ell}$. So we have
$$\delta_{\Q}(G, \ell) \ge \frac{\Delta(\ell,p)}{p} \cdot \frac{1}{9\ell} \cdot \big(\frac{1}{72 p^2\ell} \big)^{n-r} \ge \frac{1}{18\cdot 72^{n-r}} \cdot \frac{1}{p^{2n+2-r}}\cdot \frac{1}{ \ell^{n+2-r}}.$$
\end{example}

%\begin{remark}[About optimal saving]
%	The whole paper focuses on enlarging the set $\mathcal{G}(\ell)$, but does not aim for an optimal saving. There are at least several sources to improve the savings with minimal amount of efforts. 
%\end{remark}

\section{Acknowledgement}
The author is partially supported by Foerster-Bernstein Fellowship at Duke University. I would like to dedicate this paper to Prof. Nigel Boston, from whom I learnt a lot on both $p$-group theory and number theory in my graduate school. I would like to thank Dimitris Koukoulopoulos, Robert J. Lemke Oliver, Jesse Thorner, Melanie Matchett Wood and Asif Zaman for helpful conversations. I would like to thank Frank Thorne and Melanie Matchett Wood for suggestions on an earlier draft. 

%\section{More Solvable Extensions}
%I have to think what are the general descriptions here.
%\section{Examples Breaking GRH bound}
%\subsection{Elementary Abelian Extensions with High Rank}
%\subsection{More Frobenius Examples?}
%Need to think.

%\bibliographystyle{alpha}
%\bibliography{PtWs.bib}

\newcommand{\etalchar}[1]{$^{#1}$}

\Addresses
\end{document}